\documentclass[reqno, a4paper, draft]{amsart}
\usepackage{amssymb, amscd, amsfonts,  amsthm}

\usepackage[mathscr]{eucal}
\usepackage{graphicx}
\usepackage[all,poly]{xy} 
\usepackage[all]{xy} 

\newtheorem{theorem}{Theorem}[section]

\newtheorem{proposition}[theorem]{Proposition}

\newtheorem{satz}[theorem]{}

\theoremstyle{definition}
\newtheorem{dixmier}[theorem]{}

\theoremstyle{definition}
\newtheorem{definition}[theorem]{Definition}

\newcommand{\restrict}{\,{\mathbin{\vert\mkern-0.3mu\grave{}}}\,}

\newcommand{\luk}{\L u\-ka\-s\-ie\-wicz}

\newcommand{\remove}[1]{}

\DeclareMathOperator{\Mod}{\rm Mod}

\DeclareMathOperator{\flip}{\mathfrak f}

 \title[3-valued logic on the $n$-cube]
{Logic on the  $n$-cube}

\author{Daniele Mundici}

\address[D. Mundici]{Department of
Mathematics  ``Ulisse Dini'' \\
University of Florence\\
Viale Morgagni 67/A \\
I-50134 Florence \\
Italy}

\email{ mundici@math.unifi.it }

\date{\today}

\begin{document}

\thanks{2000 {\it Mathematics Subject Classification.}
Primary: 03G20    Secondary: 06D25, 06D35,
03B50, 03B53}
\keywords{Rota-Metropolis cubic logic, logics admitting inconsistency,
\luk\ logic, Post algebras of order 3, De Luca-Termini sharpening order}

\begin{abstract}
We endow the partially ordered
set   of nonempty faces of the
$n$-cube  with a distinguished 0-dimensional face
and three   operations  that naturally extend
 the Rota-Metropolis partial operations.
While the structures thus obtained  turn out to be
 term-equivalent to 
Post algebras of order 3, the inclusion
 order between faces  coincides with the
 De Luca-Termini sharpening order, and yields
 a compact  coNP-complete   logic that
 tolerates a modicum of inconsistency
 and nonmonotonicity.
\end{abstract}

\maketitle

\hfill {\it to Arnon Avron}

\bigskip

\section{Introduction: order and algebra on the faces the $n$-cube}
\label{section:intro}
For all   $n \geq 5$ the only two possible order structures
arising from the faces of regular (convex) polyhedra in euclidean
$n$-space  are those obtained from the $n$-cube and
the $n$-simplex,  \cite{cox}, \cite[p.190]{som}.

The  lattice of all faces of the $(n-1)$-simplex
 ($n=1,2,\ldots$)  can be identified with the powerset of
$\{1,\ldots,n\}$, i.e.,
with  the boolean algebra
  $\mathcal B_n$ with $2^n$ elements.  
  
For an analogous treatment of the set ${\mathcal F}_n$
of {\it nonempty }  faces of the $n$-cube, 
in   \cite{rotmet},  Rota and
Metropolis endow   
 ${\mathcal F}_n$ 
  with an operation $\sqcup$  and
  two partial operations $\sqcap,\bigtriangleup$
  as follows:
\begin{itemize}
\item[(i)]  the smallest face
$A \sqcup B$
containing  the faces $A$ and $B$,
 
\item[(ii)]  the intersection
$A \sqcap B$ of any two intersecting 
faces $A$ and $B$, 
 
\item[(iii)] the ``antipodal''
$\bigtriangleup(B,A)$ of $A$ in $B$   whenever 
$A\subseteq B$. 
The  vertices of 
$\bigtriangleup(B,A)$
are  symmetric to the vertices of $A$
 with respect to the center of $B$.
\end{itemize}
 
To give a three-valued logical interpretation of 
$\mathcal F_n$,  
Rota and Metropolis consider the set  
of all pairs $A = (A_0, A_1)$ of disjoint
subsets of $\{1,2,\ldots,n\}$, with the
understanding that  
 $A_0$ (resp., $A_1$)
is the set of coordinates where all points of   
the face  $A$  of the $n$-cube
constantly have value 0 (resp.,  value 1). 
The  operation $\sqcup$ is   given by
%
%
$(A_0,A_1) \sqcup (B_0,B_1) = (A_0 \cap B_0, A_1 \cap B_1).$
%
%
The partial operation $\sqcap$ is defined whenever
$A_0 \cap B_1 = \emptyset = A_1 \cap B_0$, by 
%
%
$(A_0,A_1) \sqcap (B_0,B_1) = (A_0 \cup B_0, A_1 \cup B_1).$
%
%
The partial operation $\bigtriangleup$ is defined whenever
$A_0 \supseteq B_0$ and $A_1 \supseteq B_1,$
by  
%
%
$\bigtriangleup((B_0,B_1), (A_0,A_1) ) =
(B_0 \cup (A_1 \setminus B_1), B_1 \cup (A_0 \setminus B_0)).$
In \cite[p.694]{rotmet}  Rota and Metropolis
write:

\smallskip
\begin{quote}
{\small
Each face $A = (A_0, A_1)$ of the $n$-cube is the result
of sampling a population $S = \{1,\ldots,n\}$, with a
view of testing the validity of a yes-no hypothesis. Here
$A_1$  and $A_0$ are the subsets of $S$ where the hypothesis
does or does not hold, respectively.
A third truth-value ``not-yet-known" can be assigned to each
element in  $S \setminus (A_0 \cup A_1)$.
Two results $A$ and $B$ of this sampling are said to be
incompatible if the two faces $A$ and  $B$  are disjoint.
}
\end{quote}

Using this representation of
$\mathcal F_n$ 
and writing for every function
$f \colon \{1,\ldots,n\}\to \{0,1/2,1\}$, 
$
 \iota(f)=(A_0,A_1)=(f^{-1}(0),f^{-1}(1)),
$
it follows that $\iota$ is a one-one correspondence
(actually, an isomorphism)
  between the set $\{0,1/2,1\}^{n}$ of such functions
  and  $\mathcal F_n$. 
We will identify via $\iota$ the two sets  $\mathcal F_n$
and  $\{0,1/2,1\}^{n}$. 
Following Rota and Metropolis,
by  a  {\it (finite)  cubic algebra} we mean  a partial 
structure  $C= (C, \sqcup, \sqcap, \bigtriangleup)$
which for some   integer 
$n \geq 1$  is isomorphic to 
$({\mathcal F}_n, \sqcup, \sqcap, \bigtriangleup)$.

 To give a logical interpretation 
  of  the    operations
 $\sqcup, \sqcap$  and $\bigtriangleup$, in 
 Section \ref{section:Post} we modify the definition of
cubic algebra  by  extending
  $\bigtriangleup(x,y)$ to the total operation
$\partial(x,y)=\bigtriangleup(x\sqcup y,y)$.
Using our identification $\mathcal F_n=\{0,1/2,1\}^n,$
we further equip every cube  with two distinguished
faces   $1/2$ and 0, where
  $1/2$ denotes  the cube itself, while  $0$
  (i.e., the constant zero function),
  denotes a distinguished 0-dimensional face  of the cube, 
 called {\it origin}.
We finally replace the partial operation $\sqcap$ by the
everywhere defined operation
$
x\wedge y=(0\sqcup x)\sqcap(0\sqcup y)\sqcap (x\sqcup y).
$
By definition, a
  {\it   Rota-Metropolis algebra,} (for short,
  {\it RM-algebra}) is  a 
structure 
 with two distinguished elements  $1/2$ and $0$ and three
everywhere defined binary operations  $\sqcup,\partial,\wedge$,
satisfying all equations satisfied by the 1-cube 
$({\mathcal F}_1, 0, 1/2, \sqcup,\partial,\wedge)$.

In Theorem \ref{theorem:post=cubic} we show that
all   operations of
  RM-algebras are definable in terms of
the  operations of   Post algebras of order 3---and
vice versa.  It follows that
the two categories of RM-algebras and
Post algebras of order 3 are 
 equivalent.
As noted in Theorem
 \ref{theorem:deluca}, 
 the inclusion order between faces, 
 when interpreted in Post  algebras, coincides
with the  De Luca-Termini ``sharpening'' or
``enhancing'' order (see \cite{del} and references therein). 

In Section \ref{section:logic}
we introduce a   
 consequence relation $\models_\Diamond $  that stands to
 the natural inclusion order between faces of cubes as the usual
 consequence relation  in Post  logic stands
 to the  natural order of Post algebras.  
We  show that the resulting logic is compact, and  
the problem $\alpha \models_\Diamond \beta$ is
coNP-complete. 
In sharp contrast with Post logic, $\models_\Diamond $
is (moderately)  inconsistency tolerant and  non-monotonic.

\section{Post algebras of order 3}
 \label{section:Post}
For background in universal algebra
  we refer the reader to
  \cite{mck}.

\begin{definition}
\label{definition:kleene} 
(\cite[\S 1]{cig})
A {\it Kleene algebra} 
is  a distributive lattice   
$$(A, 0, 1, \neg , \vee, \wedge)$$
   with smallest element  0  and largest element  1  such that  $\,\,\neg\neg x= x, \,\, \neg(x\vee y) = \neg x \wedge \neg y,\,\,$  and $\,\,x\wedge \neg x \leq  y\vee \neg y.$
\end{definition}

There are many equivalent definitions of  Post algebra of order 3
(see,e.g., \cite{baldwi, boi, tor}).  
In this paper we will adopt the following:

\begin{definition}
\label{definition:post}
(\cite[Definition 1.1]{dicpet}) 
  A {\it Post algebra of order 3} is an algebra  
  $$A = (A, 0, 1/2, 1,\neg , \nabla, \vee, \wedge)$$
    such that  $(A, 0, 1, \neg , \vee, \wedge)$  is a Kleene algebra, 
    $1/2=\neg1/2,$ and for all  
 $\,x\in A, \,\,\neg x\wedge \nabla x = \neg x\wedge  x$
 and $\neg x \vee \nabla x = 1.$
 \end{definition}
 
As noted  in \cite[p.242]{cig},
every Kleene algebra satisfies
the equation
$
\nabla(x\wedge y)=\nabla x\wedge \nabla y,
$
whence condition (iii) in \cite[Definition 1.1]{dicpet}
is redundant.

Post algebras of order 3 are 
also known as ``centered 
3-valued \L ukasiewicz algebras''.
Throughout this paper,    
{\it Post algebra} will mean ``Post algebra of order 3''.

\medskip

\noindent
{\it Example.}
Let  $\mathfrak Z$  denote the set $\{0,1/2,1\}$.
Equipping  $\mathfrak Z$ with the operations
\begin{equation}
\label{equation:post-operations}
\neg x=1-x,  \quad
\nabla x=  
 \min(1,2x),\quad
x\vee y =\max(x,y),\quad 
x\wedge y =  \min(x,y),
\end{equation}
we obtain the Post algebra  
$\mathfrak Z_{\rm Post}=
(\mathfrak Z,0,1/2,1,\neg,\nabla,\vee,\wedge).$
%
 
 %

%

\begin{theorem}
\label{theorem:post-saga}  Adopt the above notation
and terminology:
\begin{itemize}
\item[(i)]
An algebra
 $Q
  =(Q ,0,1/2, 1,\neg,\nabla,\vee,\wedge)$
   is a Post algebra  iff it satisfies all equations
satisfied by  $\mathfrak Z_{\rm Post}$
iff  it  belongs to the equational class
  $HSP(\mathfrak Z_{\rm Post})$  generated by
$\mathfrak Z_{\rm Post}$. 

\smallskip
\item[(ii)]
Fix a cardinal  $\kappa>0$. Then the
  free Post algebra on $\kappa$ generators is the
 set  $\mathfrak Z^{\mathfrak Z^\kappa}$ 
 of functions from the Tychonov cube
 $\mathfrak Z^\kappa$ to
 $\mathfrak Z,$ obtainable from the
  constant
 functions $0, 1/2, 1 $ and 
  the coordinate functions
 $(x_1,\ldots,x_\alpha,\ldots)\mapsto x_\alpha,\,\,\,($for 
 each ordinal 
 $\alpha$ with $0\leq \alpha<\kappa)$
 by pointwise application of the  operations
 of $\mathfrak Z_{\rm Post}$.
   
 \smallskip
\item[(iii)]
Up to isomorphism, every Post algebra   
   is the algebra of all continuous $\mathfrak Z$-valued
functions over some  totally disconnected compact
Hausdorff space, with the pointwise operations of
$\mathfrak Z_{\rm Post}$.
\end{itemize}
\end{theorem}

\begin{proof}
(i) and (ii)  follow from   Birkhoff theorem  \cite[4.131]{mck},
   together with   \cite[Corollary 4, p.203]{baldwi}.
For (iii) see   
\cite[Theorem 5, p.198]{baldwi}, or \cite[1.6]{dicpet}.
\end{proof}

 \medskip
The following  binary operations on $\mathfrak Z$ will 
be frequently used in this paper
(the values of $x$ are listed in the leftmost column,
those of $y$ in the top row):

\bigskip

\begin{center}
\begin{tabular}{|c|c|c|c|}
\hline
$x\sqcup y$&0&$1/2$&1\\
\hline
0&0&$1/2$&$1/2$\\
\hline
$1/2$&$1/2$&$1/2$&$1/2$\\
\hline
1&$1/2$&$1/2$&1\\
\hline
\end{tabular}
\quad\quad\quad\quad
\begin{tabular}{|c|c|c|c|}
\hline
$\partial(x,y)$&$0$&$1/2$&$1$\\
\hline
$0$&0&$1/2$&0\\
\hline
$1/2$&1&$1/2$&0\\
\hline
$1$&1&$1/2$&1\\
\hline
\end{tabular}
\end{center}

 \bigskip

\medskip

\begin{proposition} 
\label{proposition:term-eq}
With reference to  (\ref{equation:post-operations})  we have

\begin{itemize}
\item[(i)]  The operations $\neg, \nabla,  \vee$
  are definable in $\mathfrak Z$ from 
$0,1/2,\wedge,\partial$,
and so is the operation 
\begin{equation}
\label{equation:Delta}
\Delta x=\max(0,2x-1).
\end{equation}

\bigskip
\item[(ii)]   The binary operation
   $\sqcup\colon \mathfrak Z^2\to \mathfrak Z$ 
  is  definable from $0,1/2,\wedge,\partial$
  as follows:
  \begin{equation}
  \label{equation:sqcup}
x\sqcup y=(\neg\nabla y\wedge \nabla y\wedge 1/2)
\vee(\Delta y\wedge(1/2\vee\Delta x))\vee \partial(0,y).
 \end{equation}

\bigskip
\item[(iii)]   The binary operation
 $\wedge\colon \mathfrak Z^2\to \mathfrak Z$ 
  is not  definable from $0,1/2,\sqcup,\partial$.

\bigskip
  \item[(iv)] 
   The  algebras
 $\mathfrak Z_{\rm Post}
=(\mathfrak Z,0,1/2,1,\neg,\nabla,\vee,\wedge)$  and
$\mathfrak Z_{\rm RM}=
(\mathfrak Z,0,1/2, \sqcup,\partial,\wedge)$ are
term-equivalent.
In detail,  for all $x,y\in\mathfrak Z$
we have:
%
\begin{equation}
\label{partial-from-post}
\partial(x,y)
= \left(1/2 \wedge\nabla y\wedge \nabla\neg y\right)
\vee  (\Delta x\wedge\Delta y)
\vee
 (\nabla x\wedge \Delta\neg y),
\end{equation}
with  $\sqcup$  given 
by (\ref{equation:sqcup}).
Vice versa,
\begin{equation}
\label{post-from-cubic}
1=\partial\left({1}/{2},0\right),
\,\, \neg x=\partial\left({1}/{2},x\right),
\,\,\nabla x= \partial(x,0),\,
\,\,
  x\vee y=\neg(\neg x\wedge \neg y).
\end{equation}

\end{itemize}
\end{proposition}

\medskip
\begin{proof}  
(i)  It is easy to verify that
$\neg x=\partial(1/2,x)$, $\nabla x=\partial(x,0),$
$x\vee y=\neg(\neg x\wedge \neg y)$
 and $\Delta x=\neg\nabla\neg x.$

\medskip
(ii) is proved by a tedious but straightforward
verification using (i). 


\medskip
  (iii) By way of contradiction, let us suppose that
 $\wedge$ is definable.
 
 \smallskip
 \noindent{\it Case 1:}  $x\wedge y=f(x)\sqcup g(y)$ for suitable
functions $f,g\colon \mathfrak Z\to \mathfrak Z.$

Then $f(0)=g(0)=0$ and
$f(1)=g(1)=1$
  whence  $f(1)\sqcup g(0)=1/2$.
  On the other hand,  
$1 \wedge 0 = 0 \not=f(1)\sqcup g(0),$ a contradiction. 
 
 \smallskip
\noindent{\it Case 2:}  $x\wedge y=\partial(f(x),g(y))$ for suitable
functions $f,g\colon \mathfrak Z\to \mathfrak Z.$

If  $g(1)\in\{0,1\}$ then $\partial(f(1/2), g(1))\in \{0,1\}$
while $1/2\wedge 1=1/2$, a contradiction showing
that  $g(1)=1/2$.  It follows that $\partial(f(1),g(1))=1/2
\not=1\wedge 1.$
Thus 
$x \wedge y\not=\partial(f(x),g(y))$,
 another contradiction.
 
 \medskip
 (iv)  follows from
a straightforward computation.
\end{proof}


\medskip

\remove{
\begin{satz}
All unary functions  $f\colon \mathfrak Z\to \mathfrak Z$ are
obtainable from the identity function $x$, once the
  set $\mathfrak Z$ is equipped with 
the constants $0,1/2$  and the operations $\sqcup,\partial$. 
\end{satz}

\begin{proof}
We will use the following notation:
\begin{equation}
\label{equation:further-notation}
1=\partial(1/2,0),\,\,\,\vec{x}=
\partial(x,0)=\min(1,2x).
\end{equation}
Direct inspection shows that the six functions
$\neg\flip(\neg x),\,\,\,\flip(\neg x),\,\,\,\neg x,\,\,\, x,\,\,\,
\flip(x),\,\,\,\neg\flip(x) $
yield all unary one-one functions over $\mathfrak Z$. To complete
the proof  it  suffices to show  that all {\it monotonic}
functions are definable.  There are 10 such functions: in increasing 
order they are respectively given by
$
\partial(1/2,1),\,\,\, \flip\partial(x,1),\,\,\,\partial(x,\partial x,0),\,\,\,
\flip(\vec{x}), \,\,\,x,$
$\,\,\,\vec{x}, \,\,\,1/2, \,\,\,\flip\neg\flip\partial(x,1),\,\,\,
\flip\neg\flip\partial(x,0),\,\,\, \partial(1/2,0).
$
\end{proof}
}

\section{RM-algebras}
 \label{section:RM}
  Algebras in the equational class  
 $HSP(\mathfrak Z_{\rm RM})$ generated
 by $\mathfrak Z_{\rm RM}$
 are called  {\it RM-algebras.}

\begin{theorem}[Equivalent categories]
\label{theorem:post=cubic}  
With the above  stipulations we have:

\begin{itemize} 
 \item[(i)]  An algebra
$R=(R,0,1/2,\sqcup,\partial,\wedge)$
 is an RM-algebra iff it satisfies all equations
satisfied by $\mathfrak Z_{\rm RM}$. 

\smallskip
  \item[(ii)]
There is a finite set $\mathcal E$ of equations involving
the constants  $0,1/2$ and the operations 
$\sqcup,\partial, \wedge$ such that
 RM-algebras can be redefined as those
 algebras  satisfying all equations
in  $\mathcal E$. 
 
 \smallskip
  \item[(iii)] For any RM-algebra
  $B =(B,0,1/2, \sqcup,\partial,\wedge)$ let 
  $$B' 
 = (B,0,1/2,1,\neg,\nabla,\vee,\wedge)$$   be the
 algebra obtained by defining  
  $\neg,\nabla,\vee,\wedge$, and
 $1=\partial(1/2,0)$ as in 
 (\ref{post-from-cubic}).
  Then  $B'$ is a Post algebra.
  Conversely, for every Post algebra 
 $A=(A,0,1/2,1,\neg,\nabla,\vee,\wedge)$ let 
 $$A'
 =(A,0,1/2, \sqcup,\partial,\wedge)$$
  be the
 algebra obtained by defining  
 the operations $\partial,\wedge,\sqcup$ as in 
Proposition 
\ref{proposition:term-eq}(iv).
Then  $A'$ is an RM-algebra.

\smallskip
  \item[(iv)]
    The two categories of Post algebras and
   RM-algebras are  
   equivalent, and they are also equivalent to the category
   of boolean algebras.   
   \end{itemize}
 \end{theorem}

\begin{proof} 
(i) From  Birkhoff theorem \cite[4.131]{mck}.

(ii)  We can effectively write down
$\mathcal E$ starting from the  defining equations
of Post algebras (of order 3) as given
by Definition \ref{definition:post}, and  translating them into
equations for RM-algebras using  Proposition
\ref{proposition:term-eq}(iv). 
 
(iii) Follows  from Proposition
\ref{proposition:term-eq}(iv) using,
if necessary, \cite[4.140]{mck}. 

The first statement
   of (iv) follows from (iii). For the rest,
see  \cite[Theorem 8(ii), p.202]{baldwi}.
 \end{proof}


  \begin{theorem}[Representation of RM-algebras]
  \label{theorem:cubic-representation}  
  Let $A
  =(A ,0,1/2, \sqcup,\partial,\wedge)$ be an RM-algebra.
  \begin{itemize}
  
  \item[(a)]
Up to isomorphism, 
$A$ is the algebra of all continuous $\mathfrak Z$-valued
functions over some  totally disconnected compact
Hausdorff space, with the pointwise operations of
the RM-algebra 
$\mathfrak Z_{\rm RM}=(\mathfrak Z,0,1/2, \sqcup,\partial,\wedge)
=\mathcal F_1$.

\medskip 
 \item[(b)]
If  $A$ is  finite  
then  for some  $n=1,2,\ldots, $
 $A$ has $3^n$ elements, and is isomorphic to the
RM-algebra $\mathcal F_n$  of nonempty faces of the
 $n$-cube   equipped with the distinguished
 constants $0,1/2$  and operations  $\sqcup,\partial,\wedge$
 as follows:
 \begin{itemize}
 \item[(i)] $0$ is the origin, i.e., the constant function $0$;
 \item[(ii)] $1/2$ is the   cube itself, i.e.,
 the constant function $1/2$;
 \item[(iii)] $x\sqcup y$  is the smallest face containing $x$ and $y$;
 \item[(iv)] $\partial(x,y)=\bigtriangleup(x \sqcup y,y)$ is the antipodal
face of $y$ in $x\sqcup y$;
 \item[(v)] 
 $
 x\wedge y
 $
 is the intersection of the three faces
 $\,\,0\sqcup x,\,\, 0\sqcup y,\,\,\, x\sqcup y.$
 Thus, with the notation of (ii) in the Introduction, 
\begin{equation}
\label{equation:wedge}
  x\wedge y=(0\sqcup x)\sqcap(0\sqcup y)
  \sqcap (x\sqcup y).
\end{equation}
 \end{itemize}
 
   \item[(c)]   For every cardinal $\kappa>0$  the
  free RM-algebra on $\kappa$ generators
 is the set $\mathfrak Z^{\mathfrak Z^\kappa}$ 
 of all continuous  $\mathfrak Z$-valued functions
 over the Tychonov cube  $\mathfrak Z^\kappa$
equipped with the  
 constant functions $0,1/2$  and the pointwise
  operations  $\sqcup,\partial,\wedge$ of  $\mathfrak Z_{\rm RM}$
 and with the coordinate functions $(x_0,\ldots,x_\alpha,\ldots)
 \mapsto x_\alpha,$
 for each ordinal $\alpha$  with 
$0\leq\alpha<\kappa)$ as free generators.

 \end{itemize}
 \end{theorem}
 
 \begin{proof} (a) Combine Theorems
 \ref{theorem:post-saga}(iii) and
 \ref{theorem:post=cubic}(iii) with
 \cite[4.140]{mck}.
 
 (b)    (i)--(iv) are immediate.
 Then a tedious but straightforward computation
 yields (\ref{equation:wedge}).
 

 (c)   From Theorems
  \ref{theorem:post-saga}(ii) and \ref{theorem:post=cubic}(iii).  
 \end{proof}
 
 As a particular case of (iv) in the above theorem,  
 $\partial(x,0)$  is 
 the vertex  of $x$ farthest from the origin, where
   the  distance of a vertex
$\,v\,$ from the origin is 
 the number of edges in 
a shortest path leading from $0$ to $v$.

 \section{The natural inclusion order between faces}
 The relationships between the lattice operation $\wedge$  and
 the Rota-Metropolis partial operation $\sqcap$
 are deeper than what is shown 
 in (\ref{equation:wedge}).
 To see this, proceeding as in 
Theorem  \ref{theorem:post=cubic}(iii),
  we first equip every RM-algebra
 $A=(A,0,1/2,\sqcup,\partial,\wedge)$
  with the derived constant 1 and
 operations  $\neg,\nabla,\vee$ as follows:
\begin{equation}
\label{equation:superimpose-post}
1=\partial\left(1/2,0\right),\quad
 \neg x=\partial\left(1/2,x\right),
\quad\nabla x= \partial(x,0),\,
\quad
x\vee y=\neg(\neg x\wedge\neg y).
\end{equation}

\begin{definition}
\label{definition:compatible}
We  say that two elements  $a,b\in A$ are
{\it compatible} if there is $c\in A$  such that
$c\sqcup a=a$  and $c\sqcup b=b.$  Otherwise,
$a,b$ are  {\it incompatible}.
\end{definition}

 \begin{proposition}
 \label{proposition:min-is-sqcap}
 If $a$  and $b$ are compatible elements of the RM-algebra
 $\mathcal F_n$  of nonempty faces of the $n$-cube then
 their infimum  $a\sqcap b=a\cap b$  is given by
 \begin{equation}
 \label{equation:cap-curly}
 a\sqcap b=
 \left(1/2\wedge\nabla(a\wedge\neg a)
 \wedge \nabla(b\wedge \neg b)\right)
 \vee \neg\nabla(\neg a \wedge\neg b).
 \end{equation}
 \end{proposition}
 

 \begin{proof}   
  $\mathcal F_n$ is the RM-algebra of all functions
  $f\colon \{1,\ldots,n\}\to\mathfrak Z$ with the
  operations $\mathfrak Z_{\rm RM}.$ One now verifies
(\ref{equation:cap-curly}) for each  $i=1,\ldots,n$
without difficulty. 
\end{proof}


\begin{proposition} 
\label{proposition:inclusion}
Let $A=(A,0,1/2,\sqcup,\partial,\wedge)$ be an RM-algebra.  
Let  the binary relation $\sqsubseteq$ on $A$
be given by stipulating that, for all $a,b\in A,\,\,\,\,$
$a\sqsubseteq b$ iff $a\sqcup b= b$.
Recalling the notation of (\ref{equation:superimpose-post}) we have:

\begin{itemize}
\item[(i)]  $a\sqsubseteq b\,\,$ iff $\,\,\partial(a,0)\sqsubseteq b$ and 
$\partial(a,1)\sqsubseteq b$.

\smallskip
\item[(ii)]  Suppose $c\in A$ is {\rm boolean},  i.e., $c=\nabla c$.
Then $c\sqsubseteq b$  iff   $\neg c\sqcup b=1/2.$

\smallskip
\item[(iii)]  
$a\sqsubseteq b\,\,$  iff $\,\,\neg\partial(a,0) \sqcup b = 1/2\,\,$
and
$\,\,\neg\partial(a,1) \sqcup b = 1/2.$
\end{itemize}
\end{proposition}

 \begin{proof}
By Theorem
 \ref{theorem:cubic-representation}(a), 
 for some totally disconnected
compact Hausdorff space $X$,
$A$ is the  RM-algebra  of all continuous functions
  $f\colon X\to\mathfrak Z$ with the pointwise
  operations of $\mathfrak Z.$  
  The pointwise verification of 
(i)-(ii)  is now immediate. 
(iii) is proved by a tedious but straightforward calculation.
\end{proof}


\begin{theorem}
\label{theorem:deluca}
Given elements $f$ and $g$ in a Post algebra  $A$ of
continuous functions on a boolean space $X$ 
as in Theorem \ref{theorem:post-saga}(iii),
we say that  $f$ is {\em sharper}  than $g$, and write
$f\preceq g$, iff for each  $x\in X$ we either have
  $f(x) \leq g(x)\leq \neg g(x)$
 or $f(x)\geq g(x) \geq \neg g(x)$. 
 This is   the (De Luca-Termini)
 {\rm sharpening order \cite{del}.}

  We then have:

\begin{itemize} 
\smallskip
\item[(i)]  $\preceq$ equips $A$ with a partial order relation.

\smallskip
\item[(ii)] An element $p\in A$  is
  $\preceq$-minimal iff
  it is  boolean.

\smallskip
\item[(iii)]
 The  partial order    $\sqsubseteq$ on the 
RM-algebra
$\mathcal F_n=(\mathcal F_n,0,1/2,\sqcup,
\partial,\wedge)$ given by inclusion 
between nonempty faces of the $n$-cube  
coincides with  the partial order $\preceq$ on the Post 
algebra  $(\mathcal F_n,0,1/2,1,\neg,\nabla,\vee,\wedge)$
of Theorem \ref{theorem:post=cubic}(iii). 
\end{itemize} 
\end{theorem}

\begin{proof}
A tedious but straightforward verification.
\end{proof}

 \section{The underlying logic of RM-algebras}
 \label{section:logic}
\subsection*{Introducing RM-logic}
While by Theorem
\ref{theorem:post=cubic},
 RM-algebras are an inessential variant of Post algebras
(of order 3), in this section we will introduce
a consequence relation
 arising from the De Luca-Termini sharpening order  
  $\preceq\,\,=\,\,\sqsubseteq$  of Theorem \ref{theorem:deluca}.
The resulting logic turns out to be
 sharply different from Post logic.

For  $\mathcal X=\{X_1,X_2,\ldots,X_\alpha,\ldots\}$ a fixed but otherwise arbitrary (possibly uncountable)
nonempty  set of variable
symbols, the set $\mathsf{FORM}_{\mathcal X}$ of formulas
is constructed 
 in the usual way by finitely many applications of the
 connectives
   $\sqcup,\partial,\wedge$ starting from the variables
   of $\mathcal X$
and the  constant symbols  $0$ and $1/2.$

A {\it valuation}  is a function
 $V\colon 
\mathsf{FORM}_{\mathcal X}\to\mathfrak Z$  
that assigns value $1/2$ to the symbol 
$1/2$,  value 
0 to the symbol
0, and for each binary connective $\ast\in \{\sqcup,\partial,
\wedge\}\,\,\,$ satisfies the identity
$
V(\phi\ast\psi)=V(\phi)\ast V(\psi).
$
Since $V$ is uniquely determined by its restriction
$v=V\restrict\mathcal X$, and $v$ ranges over
all elements of the set $ 
\mathfrak Z^{\mathcal X}$,   then every 
$\phi\in \mathsf{FORM}_\mathcal X$ determines the function
$\hat\phi\colon \mathfrak Z^{\mathcal X}
\to \mathfrak Z$ given by
$
\hat\phi(v)=V(\phi)  \quad\mbox{for all \,\,\,} v\in \mathfrak Z^{\mathcal X}. 
$

 In particular, for each
 $v\in  \mathfrak Z^{\mathcal X}$ and variable symbol
 $X_\alpha \in\mathcal X,$
 \begin{equation}
 \label{equation:trit}
 \widehat{X_\alpha}(v)=v_\alpha,
 \end{equation}
 so that $ \widehat{X_\alpha}$ is the $\alpha$th coordinate function
 on  $ \mathfrak Z^{\mathcal X}$.

Given formulas $\phi,\psi\in  \mathsf{FORM}_\mathcal X$
we write $\phi\equiv_\Diamond\psi$ (read:  $\phi$ is
{\it equivalent} to $\psi$) if  $\hat\phi=\hat\psi.$
We will tacitly identify $\hat\phi$ with the equivalence
class $\phi/\negthickspace\equiv_\Diamond$.
The set $ \mathsf{FORM}_\mathcal X/\negthickspace\equiv_\Diamond$  of equivalence classes
is naturally equipped with the distinguished elements
 0 and $1/2$
(respectively for the constant functions  0 and  $1/2$
over $ \mathfrak Z^{\mathcal X}$),
as well as with the operations  $\sqcup, \partial, \wedge,$
where $\widehat{\phi\ast\psi}=\hat\phi\ast\hat\psi$ with
the pointwise operation  $\ast\in \{\sqcup,\partial,
\wedge\}$  on  $\mathfrak Z$.
By   abuse of notation, the resulting RM-algebra  
$
\{\hat\phi\mid\phi\in \mathsf{FORM}_\mathcal X\}
$
will be denoted 
 $\mathsf{FORM}_\mathcal X/\negthickspace\equiv_\Diamond$.
 
 \medskip

\begin{proposition}
\label{proposition:domenica}
For any,
possibly uncountable,
 set $\mathcal X \not= \emptyset$ of variables
and formula  $\phi\in \mathsf{FORM}_{\mathcal X}$,
let us equip  $\mathfrak Z^{\mathcal X}$ with the  
product  topology of the discrete set   $\mathfrak Z$. It 
 follows that
$\hat\phi$ is continuous. 
Further, 
 $\mathsf{FORM}_\mathcal X/\negthickspace\equiv_\Diamond$ 
is (isomorphic to)
 the free RM-algebra  over
the free generating set 
$\{X/\negthickspace\equiv_\Diamond\mid X\in \mathcal X\}$.
\end{proposition}

\begin{proof} 
The first statement follows by induction on the
number of connectives in $\phi.$
The second is essentially a reformulation of Theorem
  \ref{theorem:cubic-representation}(c).  
 \end{proof}

For any  $\Theta\subseteq 
  \mathsf{FORM}_{\mathcal X}$
 and $\phi\in \mathsf{FORM}_{\mathcal X}$
 we  say that $\Theta$ is {\em incompatible} if
 there is a valuation $v\in \mathfrak Z^{\mathcal X}$ and 
 formulas $\theta_1,\theta_2\in\Theta$
 such that  $\hat\theta_1(v)=1-\hat\theta_2(v)$.
 Otherwise,  $\Theta$ is {\em compatible}. 
 
 A moment's reflection shows that 
 $\theta_1$  and $\theta_2$  are compatible
 iff $\hat\theta_1$  and $\hat\theta_2$ are compatible
 in the sense of Definition 	\ref{definition:compatible}.

 \begin{definition}[RM-logic, defined via its
 consequence relation]
 \label{definition:culmine}
 We say that $\phi$ is a {\it consequence} of $\Theta,$
 and we write  $\Theta\models_\Diamond \phi$, according
 to the following stipulation:
 
 \smallskip
 \begin{itemize}
\item  If $\Theta$ is incompatible then
 every formula $\psi$ is a consequence of $\Theta$.

 \smallskip
  \item If $\Theta$ is compatible then  $\phi$ is a consequence
 of $\Theta$ iff
\begin{equation}
\label{equation:consequence}
 \boxed{
\forall v\in \mathfrak Z^{\mathcal X}\,\,\,
\exists
 \theta\in \Theta\cup\{{1}/{2}\} \mbox{ such that } \hat\theta(v)\sqsubseteq
 \hat\phi(v), \,\,\mbox{\rm i.e.,}\,\,\, \hat\theta(v)\sqcup
 \hat\phi(v)= \hat\phi(v).
 }
\end{equation}
 \end{itemize}
 In particular, 
 $\emptyset \models_\Diamond \phi$  iff
  $1/2 \models_\Diamond \phi$ iff $\hat\phi$ is the constant
  function $1/2$ over  $\mathfrak Z^{\mathcal X}$.
In this case we  write
 $\models_\Diamond \phi$ instead of    $\emptyset \models_\Diamond \phi$, and  say that 
 $\phi$  is a {\em tautology}.  
 If  $\Theta=\{\theta\}$ is a singleton then for any formula $\psi$ we write
 $\theta \models_\Diamond \psi$ instead of
  $\{\theta\}\models_\Diamond \psi$. 
   \end{definition}

\medskip
If  $\mathcal X\subseteq \mathcal Y$  then
$\mathsf{FORM}_\mathcal X\subseteq 
  \mathsf{FORM}_\mathcal Y$, and one might wonder
whether given  $\Theta\subseteq \mathcal X$ 
and $\phi\in \mathcal X$ we should write
$\Theta\models_{\Diamond,\mathcal X} \phi$
and
$\Theta\models_{\Diamond,\mathcal Y} \phi$
to  distinguish
between  $\Theta\models_\Diamond \phi$
in  $\mathsf{FORM}_\mathcal X$  and 
$\Theta\models_\Diamond \phi$ in 
$\mathsf{FORM}_\mathcal Y$.
The following result shows that no such
notational precaution is necessary; its proof is an
immediate consequence of the definition:

	    \begin{proposition}
	    Suppose $\mathcal X\subseteq \mathcal Y$,
$\Theta\subseteq \mathsf{FORM}_\mathcal X$ 
and $\phi\in \mathsf{FORM}_\mathcal X$. Then 
	    $
	    \Theta\models_{\Diamond,\mathcal X} \phi
\,\,\,\mbox{\rm iff}\,\,\,
	    \Theta\models_{\Diamond,\mathcal Y} \phi.
	    $
\end{proposition}

\begin{proposition}
\label{proposition:anti-deduction}
For    any formula  $\phi$ the
following conditions are equivalent:
\begin{itemize}
\item[(i)] $\phi$ is a tautology;
\item[(ii)]  both  $\,\,0\models_\Diamond \phi\,\,$  and $\,\,\partial(1/2,0)
\models_\Diamond \phi$;
\item[(iii)]  $\alpha\sqcup\neg\alpha\models_\Diamond \phi\,$ for some 
formula $\,\alpha$;
\item[(iv)]  $\beta\models_\Diamond \phi\,$ for 
every formula  $\,\beta$.
\end{itemize}
\end{proposition} 

\begin{proof} Trivial.
\end{proof}

\begin{proposition}
\label{proposition:two-tautologies}
For any two formulas $\alpha,\beta\in 
\mathsf{FORM}_\mathcal X$ the following
conditions are equivalent:
\begin{itemize}
\item[(i)] both $\alpha$ and $\beta$ are tautologies;
\item[(ii)] $\alpha\wedge\neg\alpha\wedge
\beta\wedge\neg\beta$ is a tautology.
\end{itemize}
\end{proposition}

\begin{proof} Using Proposition
\ref{proposition:domenica}, let us identify
$\mathsf{FORM}_\mathcal X/\negthickspace\equiv_\Diamond$
with the free RM-algebra over the free generating
set  $\mathcal X/\negthickspace\equiv_\Diamond,$  
given by 
Theorem  \ref{theorem:cubic-representation}(c).
Trivially, for every $v\in \mathfrak Z^\mathcal X$,
$\,\,\,
\hat\alpha(v)\wedge\neg\hat\alpha(v)\wedge
\hat\beta(v)\wedge\neg\hat\beta(v)=1/2
\mbox{ \,\,iff  \,\, $\hat\alpha(v)=\hat\beta(v)=1/2$.}
$
 \end{proof}

\begin{theorem}[Compactness]
\label{theorem:compactness}
Let  $\Theta\subseteq \mathsf{FORM}_\mathcal X$
 be an infinite set of formulas and $\phi
 \in  \mathsf{FORM}_\mathcal X$. 
Then the following are equivalent:
\begin{itemize}
\item[(i)] 
$\Theta\models_\Diamond \phi.$ 
\item[(ii)]  
$\{\theta_1,\ldots,\theta_k\}\models_\Diamond \phi\,\,\,$
for some
$\theta_1,\ldots,\theta_k\in \Theta$.
\end{itemize}
\end{theorem}

\begin{proof}  In case $\Theta$ is incompatible,
both sides are true (actually, (ii) holds with
$k=2$) and hence they are equivalent. 

 \medskip
Now suppose $\Theta$ is compatible.

 \medskip
\noindent
 (ii)$\Rightarrow $(i)  Immediate by Definition \ref{definition:culmine}.

  
 \medskip
\noindent
 (i)$\Rightarrow $(ii) 
We reformulate  (\ref{equation:consequence}) in
Definition \ref{definition:culmine} as follows:
For  every valuation $v\in \mathfrak Z^\mathcal X$,
 
 \smallskip
\begin{itemize}
\item[(a)] If   $\hat\phi(v)=1$  then there is a formula
$\alpha_v \in \Theta$  with $\widehat\alpha_v(v)=1.$

\smallskip
\item[(b)] If  $\hat\phi(v)=0$  then there is a formula
$\beta_v \in \Theta$  with $\widehat\beta_v(v)=0.$

\smallskip
\noindent
whence

\smallskip
\item[(c)] If for every $\theta\in\Theta$ we have
 $\hat\theta(v)=1/2$ then 
$\hat\phi(v)=1/2$.
 \end{itemize}
For each valuation $v$  the function $\widehat\alpha_v$
is continuous  (by Proposition \ref{proposition:domenica})
and hence  $\widehat\alpha_v$
  has value 1 on a clopen neighbourhood  $\mathcal N_v\ni v.$
Letting  $v$ range over $\hat\phi^{-1}(1)$, we see that
 the
 clopen set $\hat\phi^{-1}(1)$  is covered by the
 family of neighbourhoods
 $\mathcal N_v$. The compactness of $\mathfrak Z^\mathcal X$ 
 yields formulas 
 $\alpha_1,\ldots,\alpha_h\in\Theta$ such that
 \begin{itemize}
 \item[(a')] If   $\hat\phi(v)=1$  then there is  $i\in\{1,\ldots,h\}$
such that  $\widehat\alpha_i(v)=1.$
 \end{itemize}
Similarly, there are
$\beta_1,\ldots,\beta_k\in\Theta$ such that 
 \begin{itemize}
 \item[(b')] If   $\hat\phi(v)=0$  then there is
 then there is  $j \in\{1,\ldots,k\}$
such that  $\widehat\beta_j(v)=0.$
 \end{itemize}
 
The compatibility of $\Theta$ ensures that
its subset $\{\alpha_1,\ldots,\alpha_h,\beta_1,\ldots,\beta_k\}$
(is compatible and)  satisfies (ii).
\end{proof}

\medskip

\medskip
 \subsection*{Finite sets of premises}
 
We now consider the special case when
the set $\mathcal X$  of variables is
finite.
We write $\mathsf{FORM}_m$ instead of
$\mathsf{FORM}_{\{X_1,\ldots,X_m\}}$
(as well as instead of
$\mathsf{FORM}_{\{X_0,\ldots,X_{m-1}\}}$).

 \bigskip
 
 Recalling  Theorems \ref{theorem:post=cubic} and
\ref{theorem:cubic-representation} 
and  we immediately
have:

\begin{proposition}[Representation]
\label{proposition:logical-representation}
Every  function
$f\colon \mathfrak Z^m\to \mathfrak Z$ has the 
form $f=\hat\psi$  for some formula
$\psi\in \mathsf{FORM}_m$.  
Further, 
 $\mathsf{FORM}_m/\negthickspace\equiv_\Diamond$
 is isomorphic  to  the free RM-algebra
over the free generating set
 $\{X_1/\negthickspace\equiv_\Diamond,
 \ldots,X_m/\negthickspace\equiv_\Diamond\}.$
Thus
  $\mathsf{FORM}_m/\negthickspace\equiv_\Diamond$
   is isomorphic to the RM-algebra  $\mathcal F_{3^m}$
    of nonempty faces
 of the $3^m$-cube, with the operations
 $\sqcup,\partial,\wedge$ of  
Theorem \ref{theorem:cubic-representation}(b). 
\end{proposition}

\medskip
Following Rota and Metropolis
(see (ii) in the Introduction), 
the  partial binary operation $\sqcap\subseteq \mathfrak Z\times
\mathfrak Z$ is defined by 

\medskip

\begin{center}
\begin{tabular}{|c|c|c|c|}
\hline
$x\sqcap y$&0&$1/2$&1\\
\hline
0&0&$0$&$\mbox{\tiny undefined}$\\
\hline
$1/2$&$0$&$1/2$&$1$\\
\hline
1&$\mbox{\tiny undefined}$&$1$&1\\
\hline
\end{tabular}
\end{center}

\bigskip 
 The following result links  the consequence
 relation  $\models_\Diamond$
with the natural inclusion order between
the  faces of the $3^m$-cube:

\begin{proposition} 
\label{inclusion-finite}
  Suppose  $\Theta\subseteq \mathsf{FORM}_m$  is
compatible and $\phi\in \mathsf{FORM}_m.$   

\begin{itemize}
\item[(i)]
$\{\theta_1,\ldots,\theta_k\}\models_\Diamond\phi
\,\,\,\mbox{ iff }\,\,\,
(\hat\theta_1\sqcap\ldots\sqcap \hat\theta_k) \sqcup \hat\phi=\hat\phi,$
with the pointwise operations  $\sqcup$ and $\sqcap$
on $\mathfrak Z$.

\medskip
\item[(ii)]
In particular,  writing
${ \mathsf{FORM}_m}/{\equiv_\Diamond}
=
\{\hat\phi\mid\phi\in  \mathsf{FORM}_m\}
= \mathcal F_{3^m},
$
it follows that
\begin{equation}
\label{consequence=inclusion}
\theta \models_\Diamond \phi
\,\,\mbox{ iff }\,\,
\hat\theta \sqsubseteq \hat\phi
\,\,\,\mbox{ iff }\,\,\,
\hat\theta \sqcup\hat\phi= \hat\phi
\end{equation}
and hence,
\begin{equation}
\label{consequence=identity}
\hat\theta=\hat\phi
\,\,\,\mbox{ iff }\,\,\,
\theta\models_\Diamond \phi
\mbox{ and }
 \phi\models_\Diamond \theta.
 \end{equation} 
 \end{itemize}
\end{proposition}

\begin{proof}
The proof amounts to a tedious pointwise
verification  using  Proposition
\ref{proposition:logical-representation}.
\end{proof}

\medskip
 \subsection*{Complexity-theoretic issues in RM-logic}
 Mimicking (\ref{post-from-cubic})-(\ref{equation:superimpose-post}),
the  derived connectives $\neg,\nabla,\vee$
 are now introduced by stipulating that for
all  formulas  $\phi$ and $\psi $,
 \begin{equation}
 \label{abbreviation}
 \neg\phi,\,\,\,\, \nabla\phi, \,\,\,\,\phi \vee \psi
\,\, \mbox{  respectively stand for\,\, } \partial(1/2,\phi),
\,\,\, \partial(\phi,0),\,\,\,\neg(\neg\phi\wedge \neg \psi).
 \end{equation}
 The notations
 $
 \neg\hat\phi,\,\,\, \nabla\hat\phi, \,\,\,\hat\phi\vee\hat\psi
 $
 are self-explanatory in the light of Theorem 
 \ref{theorem:post=cubic}
 and Proposition \ref{proposition:logical-representation}.


\begin{proposition}
\label{proposition:one-axiom}
It is decidable whether
$\Theta=\{\theta_1,\ldots,\theta_k\}\subseteq
\mathsf{FORM}_m$ is  incompatible.
Further, there is a Turing machine which,
having in its input a compatible set  
$\Theta=\{\theta_1,\ldots,\theta_k\}\subseteq
\mathsf{FORM}_m$,
outputs a formula  $\omega\in \mathsf{FORM}_m$
such that
$
\hat\omega=\hat\theta_i\sqcap\ldots\sqcap\hat\theta_k.
$
\end{proposition}

\begin{proof}  We only prove the second statement.
It suffices to assume $k=2.$
Let $\omega$ be the formula
\begin{equation}
\label{equation:sabato}
\neg\nabla(\neg\theta_1\wedge\neg\theta_2) \vee
\left(1/2\wedge \nabla(\theta_1\wedge \neg \theta_1)    \wedge   \nabla(\theta_2\wedge \neg \theta_2) \right).
\end{equation}
Then using (\ref{equation:superimpose-post})
and (\ref{abbreviation}) one verifies  
$\hat\omega= \hat\theta_1\sqcap\hat\theta_2$.
\end{proof}


The following result reduces consequence to tautology
in RM-logic
(notation of (\ref{abbreviation})):

\begin{proposition}
\label{proposition:consequence-to-tautology} 
Let $\alpha,\beta\in \mathsf{FORM}_m.$ Then  
$\alpha\models_\Diamond\beta$
iff
$
 \models_\Diamond
(\beta\sqcup \nabla\neg \alpha)\wedge
(\beta\sqcup \neg\nabla \alpha)\wedge
\neg(\beta\sqcup \nabla\neg \alpha)\wedge
\neg(\beta\sqcup \neg\nabla \alpha).
$

Thus,  
 $\hat\alpha=\hat\beta$
iff
$
 \models_\Diamond
(\beta\sqcup \nabla\neg \alpha)\wedge
(\beta\sqcup \neg\nabla \alpha)\wedge
\neg(\beta\sqcup \nabla\neg \alpha)\wedge
\neg(\beta\sqcup \neg\nabla \alpha)
$
$
\wedge
(\alpha\sqcup \nabla\neg \beta)\wedge
(\alpha\sqcup \neg\nabla \beta)\wedge
\neg(\alpha\sqcup \nabla\neg \beta)\wedge
\neg(\alpha\sqcup \neg\nabla \beta).
$
 \end{proposition}
 \begin{proof}
 By  (\ref{consequence=identity}),
 together with
 Propositions
 \ref{proposition:inclusion} and \ref{proposition:two-tautologies}.
 \end{proof}

From \ref{proposition:consequence-to-tautology} 
we immediately get:


\begin{proposition}
\label{pre-deduction} There is a polynomial time reduction
of the consequence problem  $\alpha\models_\Diamond\beta$
to the tautology problem in RM-logic.
Also the converse reduction (trivially)  exists.
\end{proposition}

The problem $\alpha\models_\Diamond\beta$ is
as complicated as its boolean counterpart:

\begin{theorem}[coNP-completeness of RM-consequence]
\label{coNP-completeness}  
The problem $\alpha\models_\Diamond\gamma$ is
coNP-complete, and so is the tautology problem
$\models_\Diamond\tau$.
\end{theorem}

\begin{proof}
First of all, the tautology problem
$ \models_{\rm Post}\beta$    in Post logic
is coNP-complete: to see this, after noting
that the problem is in coNP,
one routinely reduces to this problem
the boolean tautology problem. 
Second, in the light of
Propositions
\ref{proposition:two-tautologies}
and  \ref{pre-deduction}
it is sufficient to deal with the tautology
problem  $\models_\Diamond\beta.$
Trivially the problem is in coNP.
To show coNP-hardness we will reduce to it
 the tautology  problem  
 in Post logic.  
 So let $\beta=\beta(X_1,\ldots,X_m)$  be an arbitrary input formula
 in Post logic. 
 Let the formula  $\beta'$ of RM-logic
  be obtained from
 $\beta$ by application of the substitutions of (\ref{post-from-cubic}).
 Observe that the map $\beta\mapsto\beta'$
 is computable in polynomial time.
Using  Proposition \ref{proposition:domenica}
from  $\beta$ we obtain  a function
 $\widehat{\beta'}\colon \mathfrak Z^m\to \mathfrak Z$.  
Let the
  function  $\flip\colon\mathfrak Z\to \mathfrak Z$ 
be defined by 
\begin{equation}
\label{equation:flip}
\flip(x) =
\partial(x,0)\sqcup\partial(\partial(0,x),0).
\end{equation}
Then  
 $\flip(0)=0,\,\,\, \flip(1/2)=1,\,\,\,\flip(1)=1/2,$
 and by Theorem
 \ref{theorem:post=cubic}(iii).
 we can write:
\begin{eqnarray*}
 \models_{\rm Post}\beta
&\Leftrightarrow&
\forall v\in \mathfrak Z^m,  \hat\beta(v)=1
\\
&\Leftrightarrow&
\forall v\in \mathfrak Z^m,  \flip(\hat\beta(v))=1/2\\
&\Leftrightarrow&
\models_\Diamond  \partial(\beta',0)\sqcup \partial(\partial(0,\beta'),0).
\end{eqnarray*}
This yields the desired reduction.
\end{proof}

\section{Closing a circle of ideas: the simplex and the cube}
\subsection*{From the $n$-simplex to boolean logic}
As already mentioned in the Introduction,
the lattice of all faces of the $(n-1)$-simplex
 ($n=1,2,\ldots$) is isomorphic to the boolean algebra
 $\mathcal B_n$ with $2^n$ elements.  To give a logical
 formalization of   $\mathcal B_n$,  one first
  prepares $m$ variable symbols
$X_1,\ldots,X_m$, where   $m$  is usually
 much smaller than $n$:   as a matter of fact,
 $m=\ulcorner \log_2 (n+1)\urcorner$ variables suffice.
Let $\mathsf{FORM}_m$ denote the
 set of boolean formulas in the variables
 $X_1,\ldots,X_m$.  Each formula  $\phi(X_1,\ldots,X_m)$
 determines the boolean function  $\hat\phi\colon\{0,1\}^m\to
 \{0,1\}$  in the usual way. 
 In particular, for each
 $i=1,\ldots,m,$ and $m$-tuple of bits  $b=(b_1,\ldots,b_m)$
 \begin{equation}
 \label{equation:variable}
 \widehat{X_i}(b)=b_i,
 \end{equation}
 so that $ \widehat{X_i}$ is the $i$th coordinate function
 on  $\{0,1\}^m$.
 Fix $n=1,\ldots,2^m$ and suppose
 $\Theta\subseteq \mathsf{FORM}_m$ 
is  satisfied by precisely
 $n$  valuations. Let $\Mod(\Theta)\subseteq \{0,1\}^m$
  be the set of such
 satisfying evaluations.
Say that
two formulas $\alpha,\beta$  are   {\it $\Theta$-equivalent,}
and write $\alpha\equiv_\Theta\beta,$ 
 iff $\Theta\models \alpha\leftrightarrow \beta.$
 In other words,  $\hat\alpha\restrict\Mod(\Theta)=\hat\beta\restrict\Mod(\Theta),$  where,
as the reader will recall, 
the symbol $\restrict$  denotes restriction.
Let  
 \begin{equation}
 \label{equation:lindenbaum}
 \mathsf{LIND}_\Theta=
\{\phi/\equiv_\Theta\mid\phi\in \mathsf{FORM}_m\}=
\{\hat{\phi}
\restrict\Mod(\Theta)\mid\phi\in\mathsf{FORM}_m\}
\end{equation}
 be the
  {\it Lindenbaum algebra of $\Theta$ (in boolean logic)} ,  i.e.,
  the boolean algebra consisting of all 
 $\equiv_\Theta$-equivalence
 classes of formulas equipped with the operations
 naturally induced by the boolean connectives.
 Equivalently, $\mathsf{LIND}_\Theta$ is the
 boolean algebras of all boolean functions
 on $\Mod(\Theta)$ equipped with the pointwise operations
 $\min,\max$ and $1-x.$  
    \begin{proposition} 
$\mathsf{LIND}_\Theta\cong
\mathcal B_n\cong\mbox{ powerset of }\{1,\ldots,n\}\cong$
boolean algebra of
faces of the $(n-1)$-simplex.  If $\tau\in  \mathsf{FORM}_m$
 is a tautology then $\mathsf{LIND}_\tau$ is
 isomorphic to the free boolean algebra over the
 free generating set $\{X_1/\negthickspace\equiv,\ldots,X_m/\negthickspace\equiv\}$ of coordinate functions of
 $\{0,1\}^m$.
 The latter in turn is isomorphic to the boolean algebra of
 faces of the $(2^m-1)$-simplex
 $\mathcal S_{2^m-1}$  (embedded in $\mathbb R^{2^m}$).
 \end{proposition}
 %
 %
 Table 1  summarizes the relationship between boolean
 logic and face lattices of simplexes.

 \begin{table}[ht]
    \label{tav:vocabolario}
    \footnotesize{
\begin{center}
\begin{tabular}{|r|l|}
\hline
{\sc boolean logic in $m$-variables} 
& \sc {$(2^m-1)$-simplex  $\mathcal S_{2^m-1}$}\\ 
\hline
\hline
the set of valuations $\{0,1\}^{\{X_1,\ldots,X_m\}}$  
& vertices of  the $(2^m-1)$-simplex  $\mathcal S_{2^m-1}$ \\ 
\hline
${\phi}/{\equiv},$ for $\phi$  a formula
in $m$ variables 
&
face of   $\mathcal S_{2^m-1}$\\
\hline
${\tau}/{\equiv},$ for $\tau$  a tautology
in $m$ variables 
&
   $\mathcal S_{2^m-1}$, the largest face\\
\hline
$X_i/\negthickspace\equiv,$ for $i=1,\ldots,m$   
&
the face of $\mathcal S_{2^m-1}$ given by
 the  vertices in $\hat{X_i}^{-1}(1)$\\
\hline
${\neg \phi}/{\equiv}$
& complementary face\\
\hline
${(\phi\vee\psi)}/{\equiv}, \,\,\, {(\phi\wedge\psi)}/{\equiv}$
& union, intersection of two faces\\ 
\hline
${(\phi\wedge\neg \phi)}/{\equiv}$
& the empty face  \\ 
\hline
free $m$-generator boolean algebra  
&
boolean algebra of faces of   $\mathcal S_{2^m-1}$\\
   \hline
$\alpha\models \beta$
&
$\alpha/\negthickspace\equiv\,\,\,\subseteq\,\,\, \beta/\negthickspace\equiv$\\
\hline
$ \theta\in \mathsf{FORM}_m$
satisfied by  $n$ valuations
&
$(n-1)$-simplex $S_\theta\subseteq
\mathcal S_{2^m-1}$\\
\hline
$\mathsf{LIND}_\theta$ 
&
boolean algebra of faces of $S_\theta$\\
\hline
valuation  satisfying $\theta$ 
& vertex  of
$S_\theta$\\ 
\hline 
$\psi/\negthickspace\equiv_\theta$   
& a face of $S_\theta$  \\  
\hline
			\end{tabular}
		\end{center}
		}
		\medskip
\caption{Boolean logic on the faces of the simplex.}
	    \end{table}

\subsection*{From the $n$-cube to RM-logic}
As explained in the Introduction,
Rota and Metropolis \cite{rotmet} envisaged cubic algebras
as the algebras of the three-valued counterpart of
boolean logic arising from the set
$\mathcal F_n$ of  nonempty faces of the
 $n$-cube  ($n=1,2,\ldots$).  
 To write down these faces,
 $m=\ulcorner \log_3 (n+1)\urcorner$ variables
 are  sufficient.  
 
 As in the case of boolean
 logic, it is convenient to define
 Lindenbaum algebras
for any nonempty (possibly uncountable) set
$\mathcal X $ of variables,
and any compatible set  
$\Theta\subseteq \mathsf{FORM}_\mathcal X$
 of formulas.
To this purpose, proceeding by analogy with the
boolean case, and recalling that $0,1\preceq 1/2$ in the
sharpening order,  we let  the compact set
$\Mod(\Theta)\subseteq \{0,1/2,1\}^\mathcal X$ be defined by
\begin{equation}
\label{equation:acca}
\Mod(\Theta)=\bigcap\{\hat\theta^{-1}({1}/{2})\mid\theta\in\Theta
\cup\{1/2\}\}.
\end{equation}
This definition is reminiscent of Definition  
 \ref{definition:culmine}, where it is stipulated that
  $\Theta$ has the same consequences as $\Theta\cup\{1/2\}$.

As in (\ref{equation:lindenbaum}), the
{\it Lindenbaum algebra (in RM-logic)}  $\mathsf{LIND}_\Theta$
 is now
defined as the quotient of 
$\mathsf{FORM}_\mathcal X$ by the   relation
$
\phi\equiv_\Theta \psi \Leftrightarrow
\hat\phi\restrict \Mod(\Theta) = \hat\psi\restrict \Mod(\Theta),$
with the RM-operations naturally induced by the connectives.
When $\Theta=\{\theta\}$  we write
$\mathsf{LIND}_\theta$ instead of $\mathsf{LIND}_{\{\theta\}}$.
When $\Theta=\emptyset$,  $\Mod(\Theta)=\Mod_{1/2}=\{0,1/2,1\}^\mathcal X.$

\begin{proposition}
If  $\Theta=\emptyset$ then $\mathsf{LIND}_\Theta=
\mathsf{LIND}_{1/2}=\mathsf{FORM}_\mathcal X/\equiv_\Diamond.$
If  $\,\,\Theta$ is finite,   say  $\Theta=
\{\theta_1,\ldots,\theta_h\}\subseteq \mathsf{FORM}_m,\,\,\,$
 $\Mod(\Theta)$  is a subset of $\{0,1/2,1\}^m$.
 If   $\Mod(\Theta)$ 
  has $n\geq 1$  elements,
  $\mathsf{LIND}_\Theta$ is isomorphic to the RM-algebra
  $\mathcal F_n$  of the $n$-cube.
 $\Mod(\Theta)$ is empty precisely when 
 the face  $\sqcap_i\hat\theta_i$ is
 a vertex of the $3^m$-cube.
\end{proposition}

\begin{proof}  From Proposition
\ref{proposition:domenica}.
\end{proof}


For completeness, in  case
$\sqcap_i\hat\theta_i$ is
 a vertex of the $3^m$-cube, we
  stipulate that   $\mathsf{LIND}_\Theta$ is the
 {\it trivial}  RM-algebra with one element  $0=1/2=1,$ alias
 the 0-cube,
 corresponding to the trivial Post algebra.

\medskip

 Table 2 sums up the machinery of RM-logic over
 finitely many variables.  
 
%
 \begin{table}[ht]
 \label{tav:vocabolario-cubico}
 \footnotesize{
\begin{center}
\begin{tabular}{|r|l|}
\hline
{\sc RM-logic} 
& \sc {faces of the cube}\\ 
\hline
\hline
the $3^m$  valuations $\{0,1/2,1\}^{\{X_1,\ldots,X_m\}}$  
& the $3^m$ dimensions
of   $3^m$-cube $\mathcal F_{3^m}$ \\ 
\hline
${\phi}/{\equiv_\Diamond},$  for $\phi$  a formula
in $m$ variables 
&
a face of   cube $\mathcal F_{3^m}$
(among $3^{3^m}$ faces)\\
\hline
${\tau}/{\equiv_\Diamond},$  for $\tau$  a tautology
in $m$ variables 
&
largest face of   cube $\mathcal F_{3^m}$, i.e., 
$\mathcal F_{3^m}$ itself\\
\hline
${\phi}/{\equiv_\Diamond},$  where  $\hat\phi(v)
\in\{0,1\}\,\,\,\forall v\in \{0,1/2,1\}^m$  
&
vertex of  cube $\mathcal F_{3^m}$\\
\hline
$X_i/\negthickspace\equiv_\Diamond,$ for $i=1,\ldots,m$   
&
 $i$th coordinate function on $\{0,1/2,1\}^{m}$\\
\hline
$\mathsf{FORM}_m/\negthickspace\equiv_\Diamond$ {\footnotesize  free $m$-generator RM-algebra}
&
{RM-algebra of faces of  $3^m$-cube}
 $\mathcal F_{3^m}$\\
 \hline  
$\theta\models_\Diamond  \psi$   
&  
${\theta}/{\equiv_\Diamond}\mbox{ is a subface of }{\psi}/{\equiv_\Diamond}$
in the $3^m$-cube  \\ 
 \hline
$ ({\phi\sqcup\psi})/{\equiv_\Diamond}$
& smallest face containing  $ {\phi}/{\equiv_\Diamond}$ and 
$ {\psi}/{\equiv_\Diamond}$\\
\hline
$ {\partial(\psi, \phi)}/{\equiv_\Diamond}$
&  the antipodal   of $ {\phi}/{\equiv_\Diamond}$  in 
$ ({\psi\sqcup \phi})/{\equiv_\Diamond}$\\ 
\hline 
${(\phi\wedge \phi)}/{\equiv_\Diamond}$
& the face ${\phi}/{\equiv_\Diamond}\wedge {\psi}/{\equiv_\Diamond}$
 \\ 
\hline
$\theta\in \mathsf{FORM}_m$ such that
$\Mod(\theta)$
 has   $n$ elements
&
 $n$-cube  $\mathcal C_\theta=\theta/\negthickspace\equiv_\Diamond$
 as a face  of the $3^m$-cube
   \\
    \hline
$\mathsf{LIND}_\theta$
&
RM-algebra of faces of $n$-cube
 \\
     \hline
$\phi/\negthickspace\equiv_\theta$
&
a face of the $n$-cube
 \\
\hline
$\theta\in \mathsf{FORM}_m$
such that $\hat\theta^{-1}(1/2)=\emptyset$  
&
 $\theta/\negthickspace\equiv_\theta$ is a vertex of the $3^m$-cube
 \\
    \hline
 \end{tabular}
\end{center}
}
		\medskip
\caption{The $n$-cube and its RM-logic.}
	    \end{table}
 %
\section{Final remarks and problems}
Intuitively, the formula $\neg\phi$ in RM-logic
means
 ``$\phi$, the other way round'', in 
 accordance  with Ramsey's view of  
$\neg\phi$  as the result of writing $\phi$  
upside down, \cite{ram}.
It follows that 
the consequence relation $\models_\Diamond$
of RM-logic has  a (limited)
 consistency tolerance  property, which Post logic
 does not have:

\begin{satz}  The pair  $\{\phi,\neg\phi\}$ is compatible
iff $\phi$ is  a tautology.  
If
$\{\phi,\neg\phi\}$ is compatible then  
$\{\phi,\neg\phi\}\models_\Diamond \psi $ iff $\psi$ is
a tautology.
\end{satz}

The disjunction connective
$\sqcup$  has no dual conjunction
$\sqcap$.  For,
$
 \neg\phi\sqcup\neg\psi\equiv_\Diamond\neg(\phi\sqcup\psi).
$
The  connective
$\wedge$  has the following
consistency tolerance and nonmonotonicity  properties,
which disappear when $\wedge$ is thought of
as conjunction in Post logic:

\begin{satz}
For every formula  $\phi$, the pair 
$\{\phi,\phi\wedge \neg\phi\}$ is incompatible
iff $\hat\phi(v)=1$ for some valuation $v$. 
In general,  the set of
consequences of  $\alpha\wedge \beta$ 
is not larger than the set of consequences of  $\alpha.$
\end{satz}

Among the derived connectives of 
RM-logic, the
 ``possibility''
connective  $\nabla$ transforms
  $\phi$ into the ``remotest 
 possibility $\nabla\phi$ (from the origin)'', where
$\nabla\phi \negthickspace\,\,\,\equiv_\Diamond
\nabla\nabla\phi \negthickspace\,\,\,\equiv_\Diamond\partial(\phi,0)$.
We also have the ``dual (nearest) possibility''
$\Delta\phi$,  defined as $\partial(\phi,\partial({1}/{2},0))$,
and satisfying
$\Delta\phi
\negthickspace\,\,\,\equiv_\Diamond
\Delta\Delta\phi\negthickspace\,\,\,\equiv_\Diamond
\neg\nabla\neg\phi.$

\medskip
Concerning implication in RM-logic, 
 the  equivalences  
$$\hat\alpha\subseteq\hat\beta
\,\,\Leftrightarrow\,\, \alpha\models_\Diamond\beta
\,\,\Leftrightarrow\,\,  \models_\Diamond
(\beta\sqcup \nabla\neg \alpha)\wedge
(\beta\sqcup \neg\nabla \alpha)\wedge
\neg(\beta\sqcup \nabla\neg \alpha)\wedge
\neg(\beta\sqcup \neg\nabla \alpha)
$$
of Propositions
  \ref{inclusion-finite}(ii)
  and
   \ref{proposition:consequence-to-tautology} 
naturally introduce a  connective
$\rightsquigarrow$ of the form
$$
\alpha\rightsquigarrow\beta=
(\beta\sqcup \nabla\neg \alpha)\wedge
(\beta\sqcup \neg\nabla \alpha)\wedge
\neg(\beta\sqcup \nabla\neg \alpha)\wedge
\neg(\beta\sqcup \neg\nabla \alpha).
$$

\begin{satz}
The $\rightsquigarrow$ connective
 satisfies:
\begin{equation}
\label{equation:MP}
\mbox{If\,\,\,} \models_\Diamond\alpha
\mbox{\,\,\,and\,\,\,}
  \models_\Diamond\alpha\rightsquigarrow\beta
  \mbox{\,\,\,then\,\,\,}
  \models_\Diamond\beta.
\end{equation}
\end{satz}


%

\begin{dixmier}  
{\sc Problems.} 

\smallskip
\begin{enumerate}  
\item  Analyze the negation connective $\neg$ in RM-logic, as well the
completeness and consistency properties of RM-logic
in the general framework of \cite{aaz, avr-neg, avr-con}.  

\smallskip

%
%
%

\item Develop the proof theory of $\models_\Diamond$
(along the lines of \cite[\S 5]{avr-bey}). 

\smallskip

\item Construct first-order RM-logic. Does first-order
 RM-logic have
a nondeterministic semantics as in \cite{avrzam}?

%
%
%
%
%
%
\end{enumerate}
\end{dixmier}



 \end{document}